\documentclass[a4paper, 10pt]{article}
\usepackage{amsmath}
\usepackage{amssymb,esint}
\usepackage{amscd}
\usepackage{xspace}
\usepackage{fancyhdr}
\newtheorem{theorem}{Theorem}[section]

\newtheorem{corollary}[theorem]{Corollary}

\newtheorem{definition}[theorem]{Definition}

\newtheorem{lemma}[theorem]{Lemma}

\newtheorem{proposition}[theorem]{Proposition}

\newenvironment{proof}[1][Proof]{\textbf{#1.} }{\hfill\rule{0.5em}{0.5em}}
{\catcode`\@=11\global\let\AddToReset=\@addtoreset
\AddToReset{equation}{section}

\AddToReset{theorem}{section}

\begin{document}
\title{Gradient estimates for singular quasilinear elliptic equations with measure data}
\author{
	{\bf Quoc-Hung Nguyen\thanks{ E-mail address: quochung.nguyen@sns.it}}\\[0.5mm]
	{\small   Scuola Normale Superiore, Centro Ennio de Giorgi, Piazza dei Cavalieri 3, I-56100
		Pisa, Italy.}\\}
\date{April 24, 2017}  
\maketitle
\begin{abstract}
In this paper,  we prove $L^q$-estimates for gradients of solutions to singular quasilinear elliptic equations with measure data
$$-\operatorname{div}(A(x,\nabla u))=\mu,$$
 in a bounded domain $\Omega\subset\mathbb{R}^{N}$, where $A(x,\nabla u)\nabla u \asymp |\nabla u|^p$,  $p\in (1,2-\frac{1}{n}]$ and $\mu$ is a Radon measure in $\Omega$.
\end{abstract}   
                  %     \tableofcontents
 \section{Introduction and main results} 
 In this article, we are concerned with the global weighted Lorentz space estimates  for gradients of  solutions to  quasilinear elliptic equations with measure data:
 \begin{equation}\label{5hh070120148}
                       \left\{
                                       \begin{array}
                                       [c]{l}%
                                       -\operatorname{div}(A(x,\nabla u))=\mu~~\text{in }\Omega,\\ 
                         u=0~~~~~~~\text{on}~~
                                                                                              \partial \Omega,
                                                                                                 \\                          
                                       \end{array}
                                       \right.  
                                       \end{equation}   where  $\Omega$ is a bounded open subset of $\mathbb{R}^{n}$, $n\geq2$,  a bounded Radon $\mu$ in $\Omega$ and  the nonlinearity  $A:\mathbb{R}^n\times \mathbb{R}^n\to \mathbb{R}^n$ is a Carath\'eodory vector valued function, i.e. $A$ is measurable in $x$ and continuous with respect to $\nabla u$ for a.e. $x$. Moreover,  $A$ satisfies 
                                       \begin{align}
                                       \label{condi1}| A(x,\xi)|\le \Lambda |\xi|^{p-1},~~~| D_\xi A(x,\xi)|\le \Lambda |\xi|^{p-2}
                                       \end{align}
                                        \begin{align}
                                       \label{condi2}  \langle D_\xi A(x,\xi)\eta,\eta\rangle\geq \Lambda^{-1} |\eta|^2|\xi|^{p-2},
                                       \end{align}
                                          for every $(\xi,\eta)\in \mathbb{R}^n\times \mathbb{R}^n\backslash\{(0,0)\}$ and a.e. $x\in \mathbb{R}^N$, where  $\Lambda$ is a  positive constant.\\
                                          In this paper, we consider singular case: 
                                          \begin{align}
                                          1<p\leq 2-\frac{1}{n}.
                                          \end{align}                   
                                          Our main result is that, for any $q>1$ and any $ w\in \mathbf{A}_q$ (the Muckenhoupt class, see below), a bounded Radon measure $\mu$ in $\Omega$, and under some additional conditions on  nonlinearity  $A$ and on the boundary of $\Omega$, then for any a (renormalized) solution $u$ of \eqref{5hh070120148}, we have  
                                          \begin{align}\label{es20}
                                          \int_{\Omega}|\nabla u|^q dx\leq C \int_{\Omega}\left( \mathbf{K}_1(\mu)\right)^{\frac{q}{p-1}}dx
                                          \end{align} 
                                          for any $q>0$,     
                                          where \begin{equation*}
                                          \mathbf{K}_1(\mu)(x):=\left\{ \begin{array}{l}
                                          \mathbf{M}_1[\mu]~~~\text{if}~~\frac{3n-2}{2n-1}<p\leq 2-\frac{1}{n},\\ 
                                          \left[\mathbf{M}_{\sigma}[|\mu|^\sigma]\right]^{\frac{1}{\sigma }}, ~~\text{ if }~1<p\leq \frac{3n-2}{2n-1}. \\ 
                                          \end{array} \right.
                                          \end{equation*}  
                                          with $\sigma>\frac{n}{p(2n-1)-2(n-1)}$,  $\mathbf{M}_\alpha$ is the fractional maximal function defined for each locally finite measure $\omega$ by
                                          \begin{align*}
                                          \mathbf{M}_\alpha[\mu](x):=\sup_{\rho>0}\frac{|\mu|(B_\rho(x))}{r^{n-\alpha}}~~\forall~~x\in \mathbb{R}^n, (\alpha\in (0,n)).
                                          \end{align*}
                               This estimate was proved in \cite{55Ph2} for regular case $p>2-\frac{1}{n}$. We recall that  $p>2-\frac{1}{n}$ is a necessary and sufficient condition for \eqref{5hh070120148} having a solution $u\in W_0^{1,1}(\Omega)$ for any Radon measure $\mu$. Similarly, $p>\frac{2n}{n+1}$ is a condition for $u\in L^1(\Omega)$. We remark that $
                               \frac{2n}{n+1}>\frac{3n-2}{2n-1}$  for any $n\geq 3$, thus we have $L^q-$estimate  \eqref{es20} without condition  $u\in L^1(\Omega)$. Furthermore,  the point-wise estimate for gradient of solutions to \eqref{5hh070120148} was obtained in \cite{55DuzaMing,Duzamin2,Mi2,55Mi0} but only for case $p>2-\frac{1}{n}$.
                               \medskip \\ 
  For our purpose, we need a condition on $\Omega$ which is expressed in the following way. We say that $\Omega$ is a $(\delta,R_0)-$Reifenberg flat domain for $\delta\in (0,1)$ and $R_0>0$ if for every $x\in\partial \Omega$ and every $r\in(0,R_0]$, there exists a system of coordinates $\{z_1,z_2,...,z_n\}$, which may depend on $r$ and $x$, so that in this coordinate system $x=0$ and that 
\begin{equation}
B_r(0)\cap \{z_n>\delta r\}\subset B_r(0)\cap \Omega\subset B_r(0)\cap\{z_n>-\delta r\}.
\end{equation}

We notice that this class of flat domains is rather wide since it includes $C^1$ domains, Lipschitz domains with sufficiently small Lipschitz constants and even fractal domains. Besides, it has many important roles in the theory of minimal surfaces and free boundary problems. This class appeared first in a work of Reifenberg (see \cite{55Re}) in the context of Plateau problem. Its properties can be found in \cite{55KeTo1,55KeTo2}.

We also require  that the nonlinearity $A$  satisfies a smallness condition of BMO type in the $x$-variable in the sense that $A(x,\zeta)$ satisfies a $(\delta,R_0)$-BMO condition for some $\delta, R_0>0$ if 
\begin{equation*}
[A]_{R_0}:=\mathop {\sup }\limits_{y\in \mathbb{R}^N,0<r\leq R_0}\fint_{B_r(y)}\Theta(A,B_r(y))(x)dx \leq \delta,
\end{equation*}         
where 
\begin{equation*}
\Theta(A,B_r(y))(x):=\mathop {\sup }\limits_{\zeta\in\mathbb{R}^N\backslash\{0\}}\frac{|A(x,\zeta)-\overline{A}_{B_r(y)}(\zeta)|}{|\zeta|^{p-1}},
\end{equation*}
and $\overline{A}_{B_r(y)}(\zeta)$ is denoted the average of $A(.,\zeta)$ over the ball $B_r(y)$, i.e,
\begin{equation*}
\overline{A}_{B_r(y)}(\zeta):=\fint_{B_r(y)}A(x,\zeta)dx=\frac{1}{|B_r(y)|}\int_{B_r(y)}A(x,\zeta)dx.
\end{equation*}                                                              
 It is easy to  see that the $(\delta,R_0)-$BMO  is satisfied when $A$ is continuous or has small jump discontinuities with respect to $x$.
We recall that                                  
A positive function $w\in L^1_{\text{loc}}(\mathbb{R}^{n})$ is called an $\mathbf{A}_{\infty}$ weight if there are two positive constants $C$ and $\nu$ such that
$$w(E)\le C \left(\frac{|E|}{|B|}\right)^\nu w(B),
$$
for all ball $B=B_\rho(x)$ and all measurable subsets $E$ of $B$. The pair $(C,\nu) $ is called the $\mathbf{A}_\infty$ constant of $w$ and is denoted by $[w]_{\mathbf{A}_\infty}$.  
It is well known that this class is the union of $\mathbf{A}_p$ for all $p\in (1,\infty)$, see \cite{55Gra}.\\  
If $w$ is a weight function belonging to $w\in \mathbf{A}_{\infty}$ and $E\subset \mathbb{R}^{n}$ a Borel set, $0<q<\infty$, $0<s\leq\infty$, the weighted Lorentz space $L^{q,s}_w(E)$  is the set of measurable functions $g$ on $E$ such that 
\begin{equation*}
||g||_{L^{q,s}_w(E)}:=\left\{ \begin{array}{l}
\left(q\int_{0}^{\infty}\left(\rho^qw\left(\{x\in E:|g(x)|>\rho\}\right)\right)^{\frac{s}{q}}\frac{d\rho}{\rho}\right)^{1/s}<\infty~\text{ if }~s<\infty, \\ 
\sup_{\rho>0}\rho \left( w\left(\{x\in E:|g(x)|>\rho\}\right)\right)^{1/q}<\infty~~\text{ if }~s=\infty. \\ 
\end{array} \right.
\end{equation*}              
Here we write $w(O)=\int_{O}w(x)dx$ for a measurable set $O\subset \mathbb{R}^{N}$.  Obviously,  
$
||g||_{L^{q,q}_w(E)}=||g||_{L^q_w(E)}
$,
thus $L^{q,q}_w(E)=L^{q}_w(E)$.               As usual, when $w \equiv 1$  we  write simply $L^{q,s}(E)$ instead of $L^{q,s}_w(E)$.              \\
In this paper, we denote   by $\mathfrak{M}_b(\Omega)$  the set of bounded Radon measures in $\Omega$.\medskip\\\\                       
We now state the main result of the paper.
                \begin{theorem} \label{101120143} Let $\mu\in \mathfrak{M}_b(\Omega)$.  Assume that $\mu\in L^\sigma(\Omega)$ for $\sigma>\frac{n}{p(2n-1)-2(n-1)}$ if $1<p\leq \frac{3n-2}{2n-1}$.  For any $w\in \mathbf{A}_{\infty}$, $0< q<\infty$, $0<s\leq\infty$ we find  $\delta=\delta(n,p,\sigma,\Lambda, q,s, [w]_{\mathbf{A}_{\infty}})\in (0,1)$ such that if $\Omega$ is  $(\delta,R_0)$-Reifenberg flat domain $\Omega$ and $[A]_{R_0}\le \delta$ for some $R_0>0$ then  for any renormalized solution $u$ of \eqref{5hh070120148} we have                            
                  \begin{equation}\label{101120144}
                              |||\nabla u|||_{L^{q,s}_w(\Omega)}\leq C ||[\mathbf{K}_1(\mu)]^{\frac{1}{p-1}}||_{L^{q,s}_w(\Omega)}.
                                       \end{equation} 
                                        Here $C$ depends only  on $n,p,\sigma,\Lambda,q,s, [w]_{\mathbf{A}_\infty}$ and $diam(\Omega)/R_0$ and   
                                        \begin{equation}\label{def}
                                        \mathbf{K}_1(\mu)(x):=\left\{ \begin{array}{l}
                                        \mathbf{M}_1[|\mu|]~~~~~~~~~\text{if}~~\frac{3n-2}{2n-1}<p\leq 2-\frac{1}{n},\\ 
                                        \left[\mathbf{M}_{\sigma}[|\mu|^\sigma]\right]^{\frac{1}{\sigma }} ~~\text{ if }~1<p\leq \frac{3n-2}{2n-1}. \\ 
                                        \end{array} \right.
                                        \end{equation}                
                                      \end{theorem}                          
               We use the notion of renormalized solutions to \eqref{5hh070120148} in Theorem \ref{101120143}, several equivalent definitions of renormalized
               solutions were given in \cite{11DMOP}. We choose the following one. If $\mu\in\mathfrak{M}_b(\Omega)$, we denote by $\mu^+$ and $\mu^-$ respectively its positive and negative parts in the Jordan decomposition. We denote by $\mathfrak{M}_0(\Omega)$ the space of measures in $\Omega$ which are absolutely continuous with respect to the $c^{\Omega}_{1,p}$-capacity defined on a compact set $K\subset\Omega$ by
  \begin{equation*}
  c^{\Omega}_{1,p}(K)=\inf\left\{\int_{\Omega}{}|{\nabla \varphi}|^pdx:\varphi\geq \chi_K,\varphi\in C^\infty_c(\Omega)\right\}.
  \end{equation*}
  We also denote $\mathfrak{M}_s(\Omega)$ the space of measures in $\Omega$ with support on a set of zero $c^{\Omega}_{1,p}$-capacity. Classically, any $\mu\in\mathfrak{M}_b(\Omega)$ can be written in a unique way under the form $\mu=\mu_0+\mu_s$ where $\mu_0\in \mathfrak{M}_0(\Omega)\cap \mathfrak{M}_b(\Omega)$ and $\mu_s\in \mathfrak{M}_s(\Omega)$.
  It is well known  that any  $\mu_0\in \mathfrak{M}_0(\Omega)\cap\mathfrak{M}_b(\Omega)$ can be written under the form $\mu_0=f-\operatorname{div}(g)$ where $f\in L^1(\Omega)$ and $g\in L^{p'}(\Omega,\mathbb{R}^n)$.
  
  For $k>0$ and $s\in\mathbb{R}$, we set $T_k(s)=\max\{\min\{s,k\},-k\}$. If $u$ is a measurable function defined  in $\Omega$, finite a.e. and such that $T_k(u)\in W^{1,p}_{loc}(\Omega)$ for any $k>0$, there exists a measurable function $v:\Omega\to \mathbb{R}^n$ such that $\nabla T_k(u)=\chi_{|u|\leq k}v$ 
  a.e. in $\Omega$ and for all $k>0$. We define the gradient $\nabla u$ of $u$ by $v=\nabla u$. 
  %%%%%%%%DEFINITION%%%%%%%%%%%%%%%%%%%%%%%%%%%%%%%%%%%%%%%%%%%%%%%%%%%%%%%%%%%%%%%%%%%%%%%%%%%%%%%%%%%%%%%%%%%%%%%%%%%%%%%%%%%%%%%%%%%%%%%%%%%%%%%%%%%%%%%%%%%%%%%%%%%%%%%%%%%%%%%%%%%%%%%%%%%%%%
  
  \begin{definition} \label{derenormalized}Let $\mu=\mu_0+\mu_s\in\mathfrak{M}_b(\Omega)$. A measurable  function $u$ defined in $\Omega$ and finite a.e. is called a renormalized solution of \eqref{5hh070120148}
  	if $T_k(u)\in W^{1,p}_0(\Omega)$ for any $k>0$, $|{\nabla u}|^{p-1}\in L^r(\Omega)$ for any $0<r<\frac{n}{n-1}$, and $u$ has the property that for any $k>0$ there exist positive nonnegative Radon measure $\lambda_k^+$ and $\lambda_k^-$ belonging to $\mathfrak{M}_0(\Omega)$, respectively concentrated on the sets $u=k$ and $u=-k$, with the property that 
  	$\mu_k^+\rightharpoonup\mu_s^+$, $\mu_k^-\rightharpoonup\mu_s^-$ in the narrow topology of measures and such that
  	\[
  	\int_{\{|u|<k\}}\langle A(x,\nabla u),\nabla \varphi\rangle
  	dx=\int_{\{|u|<k\}}{\varphi d}{\mu_{0}}+\int_{\Omega}\varphi d\lambda_{k}%
  	^{+}-\int_{\Omega}\varphi d\lambda_{k}^{-},%
  	\]
  	for every $\varphi\in W^{1,p}_0(\Omega)\cap L^{\infty}(\Omega)$.
  	
  \end{definition}
The following two properties of renormalized solutions, one can found them in \cite{11DMOP}. 
\begin{proposition}\label{pro1} For any  a renormalized solution $u$ of \eqref{5hh070120148}, we have
	\begin{align*}
	||\nabla u||_{L^{\frac{(p-1)n}{n-1},\infty}(\Omega)}+	|| u||_{L^{\frac{(p-1)n}{n-p},\infty}(\Omega)}\leq C \left[|\mu|(\Omega)\right]^{\frac{1}{p-1}}.
	\end{align*}
\end{proposition}
\begin{proposition} \label{pro2}Let $u$ be a  renormalized solution of \eqref{5hh070120148} with data $\mu\in L^1(\Omega)$. Let  $u_m$ be a solution of \eqref{5hh070120148} with data $\mu=\mu_k\in L^p(\Omega)$ such that $\mu_k\to \mu$ in $L^1(\Omega)$. Then, there exists a subsequence $(u_k)_k$  which converges 
to  $u$ in $L^{s}(\Omega)$ for all $s<\frac{(p-1)n}{n-p}$. Moreover, $\nabla u_k\to \nabla u$  in $L^q(\Omega)$ for all $q<\frac{(p-1)n}{n-1}$.
\end{proposition}
    \section{Interior estimates and boundary estimates for quasilinear equations}
In this section, we obtain certain local interior and boundary comparison
estimates that are essential to our development later. First let us consider
the interior ones. With $u\in W_{loc}^{1,p}(\Omega)$,
     For each ball $B_{2R}=B_{2R}(x_0)\subset\subset\Omega$, one considers the unique solution 
    \begin{equation*}
    w\in W_{0}^{1,p}(B_{2R})+u
    \end{equation*}
    to the following equation 
    \begin{equation}
    \label{111120146}\left\{ \begin{array}{l}
    - \operatorname{div}\left( {A(x,\nabla w)} \right) = 0 \;in\;B_{2R}, \\ 
    w = u\quad \quad on~~\partial B_{2R}, \\ 
    \end{array} \right.
    \end{equation}
    The following a variant of Gehring's lemma was proved in \cite[Theorem 6.7]{Giu}. 
    \begin{lemma} \label{111120147} Let $w$ be in \eqref{111120146}.
    	There exist  constants $\theta_1>p$ and $C$ depending only on $N,\Lambda$ such that the following estimate      
    	\begin{equation}\label{111120148}
    	\left(\fint_{B_{\rho/2}(y)}|\nabla w|^{\theta_1} dxdt\right)^{\frac{1}{\theta_1}}\leq C\left(\fint_{B_{\rho}(y)}|\nabla w|^{p-1} dx\right)^{\frac{1}{p-1}},
    	\end{equation}holds 
    	for all  $B_{\rho}(y)\subset B_{2R}$. 
    \end{lemma} 
    The next lemma gives an estimate for $\nabla u-\nabla w$. This is the main estimate of this paper.
    \begin{lemma}\label{111120149}Let $w$ be in \eqref{111120146}. Assume that $\frac{3n-2}{2n-1}<p\leq 2-\frac{1}{n}$. There holds,
    	\begin{align}\label{1111201410+}
    	\left(	\fint_{B_{2R}}|\nabla u-\nabla w|^{\gamma_0}dx\right)^{\frac{1}{\gamma_0}}\leq C \left[\frac{|\mu|(B_{2R})}{R^{n-1}}\right]^{\frac{1}{p-1}}+C\frac{|\mu|(B_{2R})}{R^{n-1}}\left(	\fint_{B_{2R}}|\nabla u|^{\gamma_0}dx\right)^{\frac{2-p}{\gamma_0}}
    	\end{align}
    	for some $\frac{2-p}{2}\leq \gamma_0<\frac{(p-1)n}{n-1}\leq 1$.  In particular, for any $\varepsilon>0$, 
    	\begin{align}\label{1111201410}
    \left(	\fint_{B_{2R}}|\nabla u-\nabla w|^{\gamma_0}dx\right)^{\frac{1}{\gamma_0}}\leq C_\varepsilon \left[\frac{|\mu|(B_{2R})}{R^{n-1}}\right]^{\frac{1}{p-1}}+\varepsilon \left(	\fint_{B_{2R}}|\nabla u|^{\gamma_0}dx\right)^{\frac{1}{\gamma_0}}.
    	\end{align}  
    \end{lemma}
\begin{proof} Set 
	\begin{equation}\label{TT}
T_{h,m}(s)=\left\{ \begin{array}{l}
T_m(s)~~~~~~~~~\text{if}~~|s|\geq 2h,\\ 
	2\operatorname{sgn}(s)(|s|-h) ~~\text{ if }~h<|s|<2h. \\ 
	0 ~~~~~~~~~\text{if}~~|s|\leq h,
	\end{array} \right.
	\end{equation}  
	for $m>2h>0$.                We have for any $\varphi\in W_0^{1,p}(B_{2R})$, 
\begin{align}\label{es0}
\int_{B_R}\langle A(x,\nabla u)-A(x,\nabla w),\nabla\varphi\rangle dx=\int_{B_R}\varphi d\mu.
\end{align}
 It is clear to see that we can choose $\varphi=T_{h,k^{1-\alpha}}(|u-w|^{-\alpha}(u-w))$ with $\alpha\in (-\infty,1)$ and $0<h< k^{1-\alpha}/2$ as test function of \eqref{es0},
 \begin{align*}
 \int_{B_{2R}\cap\{x:(2h)^{\frac{1}{1-\alpha}}<|u-w|<k\}}|u-w|^{-\alpha}g(u,w)dx\leq C k^{1-\alpha} |\mu|(B_{2R}),
 \end{align*}
 where 
 \begin{align*}
 g(u,w)=\frac{|\nabla (u-w)|^2}{(|\nabla w|+|\nabla u|)^{2-p}}.
 \end{align*}
 Using Fatou's Lemma yields,
 \begin{align}
 \int_{B_{2R}\cap\{x:|u-w|<k\}}|u-w|^{-\alpha}g(u,w)dx\leq C k^{1-\alpha} |\mu|(B_{2R}),
 \end{align}
 We now estimate $|u-w|^{-\alpha}g(u,w)$ in $L^\gamma(B_{2R})$. To do this, we employ a method in \cite{bebo} ( see also \cite{55Ph0}). Set $$\Phi(k,\lambda)=|\{x:|u-w|>k,|u-w|^{-\alpha}g(u,w)>\lambda\}\cap B_{2R}|.$$
Since $\lambda\mapsto \Phi(k,\lambda)$ is non-increasing,
\begin{align*}
\Phi(0,\lambda)&\leq \frac{1}{\lambda}\int_{0}^{\lambda}\Phi(0,s)ds\\&\leq 
\Phi(k,0)+\frac{1}{\lambda}\int_{0}^{\lambda}\Phi(0,s)-\Phi(k,s) ds
\\&=|\{x:|u-w|>k\}\cap B_{2R}|+\frac{1}{\lambda}\int_{0}^{\lambda}|\{x:|u-w|\leq k,|u-w|^{-\alpha}g(u,w)>s\}\cap B_{2R}| ds
\\&\leq k^{-\beta}||u-w||_{L^\beta(B_{2R})}^\beta+\frac{1}{\lambda}\int_{B_{2R}\cap\{x:|u-w|\leq k\}}|u-w|^{-\alpha}g(u,w) dx
\\&\leq k^{-\beta}||u-w||_{L^\beta(B_{2R})}^\beta+\frac{C k^{1-\alpha}}{\lambda} |\mu|(B_{2R}),
\end{align*} 
for any $\beta\geq 0$. Choosing $$k=\left[\frac{\lambda ||u-w||_{L^\beta(B_{2R})}^\beta}{|\mu|(B_{2R}) }\right]^{\frac{1}{1-\alpha+\beta}},$$
we find
\begin{align*}
\lambda^{\frac{\beta}{1-\alpha+\beta}}|\{x:|u-w|^{-\alpha}g(u,w)>\lambda\}\cap B_{2R}|\leq C |\mu|(B_{2R})^{\frac{\beta}{1-\alpha+\beta}} ||u-w||_{L^\beta(B_{2R})}^{\frac{\beta(1-\alpha)}{1-\alpha+\beta}}.
\end{align*}
By Holder's inequality, for $0<\gamma<\frac{\beta}{1-\alpha+\beta}$
%\begin{align*}
%||[|u-w|^{-\alpha}g(u,w)]^\gamma||_{L^{\frac{\beta}{\gamma(1-\alpha+\beta)},1}(B_{2R})}^{\frac{\beta}{\gamma(1-\alpha+\beta)}}\leq C |\mu|(B_{2R})^{\frac{\beta}{1-\alpha+\beta}} ||u-w||_{L^\beta(B_R)}^{\frac{\beta(1-\alpha)}{1-\alpha+\beta}}.
%\end{align*}
%thus, 
\begin{align}\nonumber
\int_{B_{2R}}|u-w|^{-\alpha\gamma}g(u,w)^\gamma dx&\leq C |B_{2R}|^{1-\frac{\gamma(1-\alpha+\beta)}{\beta}}||[|u-w|^{-\alpha}g(u,w)]^\gamma||_{L^{\frac{\beta}{\gamma(1-\alpha+\beta)},\infty}(B_{2R})}\\&\leq 
C R^{n-\frac{n\gamma(1-\alpha+\beta)}{\beta}} |\mu|(B_{2R})^{\gamma}||u-w||_{L^\beta(B_{2R})}^{\gamma(1-\alpha)}.\label{es1}
\end{align}
By $1<p<2$, 
\begin{align*}
|\nabla (u-w)|\leq C \left( g(u,w)^{1/p}+g(u,w)^{1/2}|\nabla u|^{\frac{2-p}{2}}\right),
\end{align*}
 it follows 
\begin{align}
&M=\int_{B_{2R}}|\nabla (u-w)||u-w|^{-\frac{\alpha}{p}}\leq C\int_{B_{2R}} |u-w|^{-\frac{\alpha}{p}}g(u,w)^{1/p}+|u-w|^{-\frac{\alpha}{p}}g(u,w)^{\frac{1}{2}}|\nabla u|^{\frac{2-p}{2}}.\label{es6}
\end{align}
By Sobolev's inequality, 
\begin{align}\nonumber
\int_{B_{2R}}|u-w|^{\frac{(p-\alpha)n}{p(n-1)}}&\leq  C \left(\int_{B_{2R}}|\nabla |u-w|^{1-\frac{\alpha}{p}}|dx\right)^{\frac{n}{n-1}}\\& =CM^{\frac{n}{n-1}}\label{es2}.
\end{align}
Using Holder's inequality, 
\begin{align}\nonumber
\int_{B_{2R}}|\nabla (u-w)|^{\frac{(p-\alpha)n}{pn-\alpha}}&\leq \left(\int_{B_{2R}}|u-w|^{\frac{(p-\alpha)n}{p(n-1)}}\right)^{\frac{\alpha(n-1)}{pn-\alpha}}\left(\int_{B_{2R}}|\nabla (u-w)||u-w|^{-\frac{\alpha}{p}}dx\right)^{\frac{n(p-\alpha)}{pn-\alpha}}\\&\leq C M^{\frac{pn}{pn-\alpha}}\label{es9}.
\end{align}
Assume that \begin{align}\label{ine1}
1/p<\frac{\beta}{1-\alpha+\beta},~~\beta=\frac{(p-\alpha)n}{p(n-1)}.
\end{align}
Thus, we can apply \eqref{es1}  to $\gamma=1/p$, we have
\begin{align}\nonumber
\int_{B_{2R}} |u-w|^{-\frac{\alpha}{p}}g(u,w)^{1/p}&\leq C R^{n-\frac{n(1-\alpha+\beta)}{p\beta}} |\mu|(B_{2R})^{1/p}||u-w||_{L^\beta(B_{2R})}^{(1-\alpha)/p}\\& \leq C R^{n-\frac{n(1-\alpha+\beta)}{p\beta}} |\mu|(B_{2R})^{1/p} M^{\frac{1-\alpha}{p-\alpha}}\label{es3}.
\end{align}
Here we used \eqref{es2} for  the last inequality.
Assume that 
\begin{align}\label{es7}
\gamma_0=\frac{(p-\alpha)n}{pn-\alpha}>\frac{2-p}{2}.
\end{align} 
Using Holder's inequality for $\frac{2\gamma_0}{2\gamma_0+p-2}$ and $\frac{2\gamma_0}{2-p}$
\begin{align}\nonumber
&\int_{B_{2R}} |u-w|^{-\frac{\alpha}{p}}g(u,w)^{\frac{1}{2}}|\nabla u|^{\frac{2-p}{2}}\\&~~~~~~\leq \left(\int_{B_{2R}} |u-w|^{-\frac{2\alpha \gamma_0}{p(2\gamma_0+p-2)}}g(u,w)^{\frac{\gamma_0}{2\gamma_0+p-2}}\right)^{\frac{2\gamma_0+p-2}{2\gamma_0}} \left(\int_{B_{2R}} |\nabla u|^{\gamma_0}\right)^{\frac{2-p}{2\gamma_0}}\label{es5}.
\end{align}
Assume that 
\begin{align}\label{ine2}
\alpha<\frac{p}{2}, ~~\frac{\gamma_0}{2\gamma_0+p-2}<\frac{\beta}{1-\frac{2\alpha}{p} +\beta}.
\end{align} 
So, by  \eqref{es1}, 
\begin{align}\nonumber
\int_{B_{2R}} |u-w|^{-\frac{2\alpha \gamma_0}{p(2\gamma_0+p-2)}}g(u,w)^{\frac{\gamma_0}{2\gamma_0+p-2}}&\leq C R^{n-\frac{n\gamma_0(1-\frac{2\alpha}{p} +\beta)}{\beta(2\gamma_0+p-2)}} |\mu|(B_{2R})^{\frac{\gamma_0}{2\gamma_0+p-2}}||u-w||_{L^\beta(B_{2R})}^{\frac{\gamma_0(1-\frac{2\alpha}{p})}{2\gamma_0+p-2}}\\& \leq C R^{n-\frac{n\gamma_0(1-\frac{2\alpha}{p} +\beta)}{\beta(2\gamma_0+p-2)}} |\mu|(B_{2R})^{\frac{\gamma_0}{2\gamma_0+p-2}}M^{\frac{\gamma_0(p-2\alpha)}{(p-\alpha)(2\gamma_0+p-2)}}\label{es4}
\end{align}
Hence, combining \eqref{es6} and \eqref{es3}, \eqref{es5} \eqref{es4} yields, 
\begin{align}\nonumber
M&\leq C R^{n-\frac{n(1-\alpha+\beta)}{p\beta}} |\mu|(B_{2R})^{1/p} M^{\frac{1-\alpha}{p-\alpha}}+\\&+\left(R^{n-\frac{n\gamma_0(1-\frac{2\alpha}{p} +\beta)}{\beta(2\gamma_0+p-2)}} |\mu|(B_{2R})^{\frac{\gamma_0}{2\gamma_0+p-2}}M^{\frac{\gamma_0(p-2\alpha)}{(p-\alpha)(2\gamma_0+p-2)}}\right)^{\frac{2\gamma_0+p-2}{2\gamma_0}} \left(\int_{B_{2R}} |\nabla u|^{\gamma_0}\right)^{\frac{2-p}{2\gamma_0}}\label{es8}
\end{align}
provided  that \eqref{ine1}, \eqref{es7} and \eqref{ine2} are satisfied. \\
Put $\alpha_0=\frac{\alpha}{p}<1/2$, so $\beta=\frac{(1-\alpha_0)n}{n-1}$
\begin{align}
1/p<\frac{\beta}{1-\alpha+\beta}~~~\Longleftrightarrow~~~p>\frac{n(2-\alpha_0)-1}{n-\alpha_0},
\end{align}
\begin{align*}
\frac{\gamma_0}{2\gamma_0+p-2}<\frac{\beta}{1-\frac{2\alpha}{p} +\beta}~~~\Longleftrightarrow~~~p>\frac{n(2-\alpha_0)-1}{n-\alpha_0},
\end{align*}
and 
\begin{align*}
\gamma_0=\frac{(p-\alpha)n}{pn-\alpha}>\frac{2-p}{2} ~~~\Longleftrightarrow~~~p>\frac{2\alpha_0(n-1)}{n-\alpha_0}.
\end{align*}
Therefore, if  \begin{align*}
p>\frac{3n-2}{2n-1}
\end{align*} 
then  \eqref{ine1}, \eqref{es7} and \eqref{ine2} hold  for any \begin{align*}
1/2>\alpha_0>\frac{-1+(2-p)n}{n-p}.
\end{align*}
Therefore, using Holder's inequality, we get from \eqref{es8} that
\begin{align}\nonumber
M&\leq C \left[R^{n-\frac{n(1-\alpha+\beta)}{p\beta}} |\mu|(B_{2R})^{1/p}\right]^{\frac{p-\alpha}{p-1}} +\\&+\left(R^{n-\frac{n\gamma_0(1-\frac{2\alpha}{p} +\beta)}{\beta(2\gamma_0+p-2)}} |\mu|(B_{2R})^{\frac{\gamma_0}{2\gamma_0+p-2}}\right)^{\frac{(2\gamma_0+p-2)(p-\alpha)}{p\gamma_0}} \left(\int_{B_{2R}} |\nabla u|^{\gamma_0}\right)^{\frac{(p-\alpha)(2-p)}{p\gamma_0}}\label{es10}.
\end{align}
It follows from \eqref{es9} and \eqref{es10} that 
\begin{align}\nonumber
&(\int_{B_{2R}}|\nabla (u-w)|^{\gamma_0})^{\frac{p-\alpha}{p\gamma_0}} \leq C \left[R^{n-\frac{n(1-\alpha+\beta)}{p\beta}} |\mu|(B_{2R})^{1/p}\right]^{\frac{p-\alpha}{p-1}} +\\&+\left(R^{n-\frac{n\gamma_0(1-\frac{2\alpha}{p} +\beta)}{\beta(2\gamma_0+p-2)}} |\mu|(B_{2R})^{\frac{\gamma_0}{2\gamma_0+p-2}}\right)^{\frac{(2\gamma_0+p-2)(p-\alpha)}{p\gamma_0}} \left(\int_{B_{2R}} |\nabla u|^{\gamma_0}\right)^{\frac{(p-\alpha)(2-p)}{p\gamma_0}}
\end{align}
At this point, using Holder's inequality, we get \eqref{1111201410} 
with $\frac{2-p}{2}< \gamma_0<\frac{(p-1)n}{n-1}\leq 1$. The proof is complete. 
\end{proof}

\begin{lemma}\label{111120149'}Let $w$ be in \eqref{111120146}. Assume that $1<p\leq \frac{3n-2}{2n-1}$ and $\mu\in L^\sigma(B_{2R})$ for  $\sigma>\frac{n}{p(2n-1)-2(n-1)}$. There holds 
	\begin{align}\label{1111201410'+}
\left(	\fint_{B_{2R}}|\nabla u-\nabla w|^{\gamma_0}dx\right)^{\frac{1}{\gamma_0}}\leq C \left[\frac{\int_{B_{2R}}|\mu|^\sigma}{R^{n-\sigma}}\right]^{\frac{1}{\sigma (p-1)}}+C \left[\frac{\int_{B_{2R}}|\mu|^\sigma}{R^{n-\sigma}}\right]^{\frac{1}{\sigma}} \left(	\fint_{B_{2R}}|\nabla u|^{\gamma_0}dx\right)^{\frac{2-p}{\gamma_0}}
	\end{align}
  for some  $\frac{2-p}{2}\leq \gamma_0<\frac{(p-1)n}{n-1}< 1$. In particular, 
  \begin{align}\label{1111201410'}
  \left(	\fint_{B_{2R}}|\nabla u-\nabla w|^{\gamma_0}dx\right)^{\frac{1}{\gamma_0}}\leq C_\varepsilon \left[\frac{\int_{B_{2R}}|\mu|^\sigma}{R^{n-\sigma}}\right]^{\frac{1}{\sigma (p-1)}}+\varepsilon \left(	\fint_{B_{2R}}|\nabla u|^{\gamma_0}dx\right)^{\frac{1}{\gamma_0}}
  \end{align}
\end{lemma}
\begin{proof} 	Put $\beta=\frac{(p-\alpha)n}{p(n-1)}$ and $\gamma_0=\frac{(p-\alpha)n}{pn-\alpha}$ with $\alpha\in (-\infty,\frac{p}{2})$. We have 
	\begin{align}
	\gamma_0>\frac{2-p}{2}.
	\end{align} 
	 It follows from  \eqref{es6}, \eqref{es2}, \eqref{es9} and \eqref{es5}  in the proof of Lemma \ref{111120149} that 
	\begin{align}\nonumber
	& M=||u-w||_{L^\beta(B_{2R})}^{\frac{p-\alpha}{p}}+||\nabla (u-w)||_{L^{\gamma_0}(B_{2R})}^{\frac{p-\alpha}{p}}\leq C\int_{B_{2R}} |u-w|^{-\frac{\alpha}{p}}g(u,w)^{1/p}\\&~~+ C\left(\int_{B_{2R}} |u-w|^{-\frac{2\alpha \gamma_0}{p(2\gamma_0+p-2)}}g(u,w)^{\frac{\gamma_0}{2\gamma_0+p-2}}\right)^{\frac{2\gamma_0+p-2}{2\gamma_0}} ||\nabla u||_{L^{\gamma_0}(B_{2R})}^{\frac{2-p}{2}},\label{es11}
	\end{align}
	where 
	\begin{align*}
	g(u,w)=\frac{|\nabla (u-w)|^2}{(|\nabla w|+|\nabla u|)^{2-p}}.
	\end{align*}
	 As in the proof of Lemma \eqref{111120149}, 
	\begin{align}
	\int_{B_{2R}\cap\{x:|u-w|<k\}}|u-w|^{-\alpha}g(u,w)dx\leq C \int_{B_{2R}}|T_{k^{1-\alpha}}(|u-w|^{-\alpha}(u-w))| |\mu|dx
	\end{align}
Let $\sigma\in (1,\min\{\frac{\beta}{\beta-1+\alpha},\frac{n}{p}\})$ and $\sigma'=\frac{\sigma}{\sigma-1}$. So, $\sigma'(1-\alpha)>\beta$.\\
Using Holder's inequality and the fact that
$$T_{k^{1-\alpha}}(|u-w|^{1-\alpha})|\leq k^{1-\alpha-\frac{\beta}{\sigma'}}|u-w|^{\frac{\beta}{\sigma'}},$$
we have
\begin{align*}
&\int_{B_{2R}\cap\{x:|u-w|<k\}}|u-w|^{-\alpha}g(u,w)dx\leq C ||T_{k^{1-\alpha}}(|u-w|^{1-\alpha})||_{L^{\sigma'}(B_{2R})} ||\mu||_{L^\sigma(B_{2R})}\\& ~~~\leq Ck^{1-\alpha-\frac{\beta}{\sigma'}} ||u-w||_{L^{\beta}(B_{2R})}^{\frac{\beta}{\sigma'}} ||\mu||_{L^\sigma(B_{2R})}.
\end{align*}
As the proof of \eqref{es1}, we also get that  for $0<\gamma<\frac{\beta}{1-\alpha-\frac{\beta}{\sigma'}+\beta}$
	\begin{align}
	\int_{B_{2R}}|u-w|^{-\alpha\gamma}g(u,w)^\gamma dx\leq 
	C R^{n-\frac{n\gamma(1-\alpha-\frac{\beta}{\sigma'}+\beta)}{\beta}}  ||\mu||_{L^\sigma(B_{2R})}^{\gamma}||u-w||_{L^\beta(B_{2R})}^{\gamma(1-\alpha)}.\label{es12}
	\end{align}
	Assume that \begin{align}\label{es15}
	1/p<\frac{\beta}{1-\alpha-\frac{\beta}{\sigma'}+\beta}.
	\end{align}
	Thus, we can apply \eqref{es12}  to $\gamma=1/p$, we have
	\begin{align}
	\int_{B_{2R}} |u-w|^{-\frac{\alpha}{p}}g(u,w)^{1/p}&\nonumber\leq C R^{n-\frac{n(1-\alpha-\frac{\beta}{\sigma'}+\beta)}{p\beta}}  ||\mu||_{L^\sigma(B_{2R})}^{1/p}||u-w||_{L^\beta(B_{2R})}^{(1-\alpha)/p}\\&\leq C R^{n-\frac{n(1-\alpha-\frac{\beta}{\sigma'}+\beta)}{p\beta}}  ||\mu||_{L^\sigma(B_{2R})}^{1/p}M^{\frac{1-\alpha}{p-\alpha}}\label{es13}
	\end{align}
Assume that
	\begin{align}\label{es15'}
	\alpha<\frac{p}{2},~~~ \frac{\gamma_0}{2\gamma_0+p-2}<\frac{\beta}{1-\frac{2\alpha}{p}-\frac{\beta}{\sigma'}+\beta}.
	\end{align} 
	So, by \eqref{es12}
	\begin{align}
	\int_{B_{2R}} |u-w|^{-\frac{2\alpha \gamma_0}{p(2\gamma_0+p-2)}}g(u,w)^{\frac{\gamma_0}{2\gamma_0+p-2}}&\leq C R^{n-\frac{n\gamma_0(1-\frac{2\alpha}{p} -\frac{\beta}{\sigma'}+\beta)}{\beta(2\gamma_0+p-2)}} ||\mu||_{L^\sigma(B_{2R})}^{\frac{\gamma_0}{2\gamma_0+p-2}}M^{\frac{\gamma_0(p-2\alpha)}{(p-\alpha)(2\gamma_0+p-2)}}\label{es16}.
	\end{align}
	Hence, combining \eqref{es11} and \eqref{es13}, \eqref{es16} yields, 
	\begin{align}\nonumber
	&M\leq CR^{n-\frac{n(1-\alpha-\frac{\beta}{\sigma'}+\beta)}{p\beta}}  ||\mu||_{L^\sigma(B_{2R})}^{1/p}M^{\frac{1-\alpha}{p-\alpha}}\\&~~+ C\left(R^{n-\frac{n\gamma_0(1-\frac{2\alpha}{p} -\frac{\beta}{\sigma'}+\beta)}{\beta(2\gamma_0+p-2)}} ||\mu||_{L^\sigma(B_{2R})}^{\frac{\gamma_0}{2\gamma_0+p-2}}M^{\frac{\gamma_0(p-2\alpha)}{(p-\alpha)(2\gamma_0+p-2)}}\right)^{\frac{2\gamma_0+p-2}{2\gamma_0}} ||\nabla u||_{L^{\gamma_0}(B_{2R})}^{\frac{2-p}{2}}\label{es17}
	\end{align}
	provided  that \eqref{es15} and \eqref{es15'} are satisfied. \\
	Put $\alpha_0=\frac{\alpha}{p}<1/2$, so $\beta=\frac{(1-\alpha_0)n}{n-1}$
	\begin{align*}
	1/p<\frac{\beta}{1-\alpha-\frac{\beta}{\sigma'}+\beta}~~~\Longleftrightarrow~~~p>\frac{n(1+\frac{1-\alpha_0}{\sigma})-1}{n-\alpha_0},
	\end{align*}
	and
	\begin{align*}
	\frac{\gamma_0}{2\gamma_0+p-2}<\frac{\beta}{1-\frac{2\alpha}{p}-\frac{\beta}{\sigma'}+\beta}~~~\Longleftrightarrow~~~p>\frac{n(1+\frac{1-\alpha_0}{\sigma})-1}{n-\alpha_0}.
	\end{align*}
	Therefore, if  \begin{align*}
\sigma>\frac{n}{p(2n-1)-2(n-1)} ~~\left(\Longleftrightarrow~~~ 	p>\frac{n(2+\frac{1}{\sigma})-2}{2n-1}~~\right),
	\end{align*} 
	then   \eqref{es15} and \eqref{es15'} hold  for any \begin{align*}
	1/2>\alpha_0>\frac{-1+(1+\frac{1}{\sigma}-p)n}{\frac{n}{\sigma}-p}.
	\end{align*}
	Therefore, the result follows from \eqref{es17} and Holder's inequality.  The proof is complete. 
\end{proof}
\begin{proposition} \label{inter} Assume that $\mu\in L^\sigma(\Omega)$ for $\sigma>\frac{n}{p(2n-1)-2(n-1)}$ if $1<p\leq \frac{3n-2}{2n-1}$. Let $\gamma_0$ be as in Lemma \ref{111120149} and Lemma  \ref{111120149'}. There exists $v\in W^{1,p}(B_R)\cap W^{1,\infty}(B_{R/2})$ such that for any $\varepsilon>0$, 
	\begin{align}
	||\nabla v||_{L^\infty(B_{R/2})}\leq C \left[\mathbf{T}(\mu, B_{2R})\right]^{\frac{1}{p-1}}+C (\fint_{B_{2R}}|\nabla u|^{\gamma_0})^{1/\gamma_0},
	\end{align}
	and
	\begin{align}
	(\fint_{B_{R}}|\nabla u-\nabla v|^{\gamma_0}dx)^{\frac{1}{\gamma_0}}\leq C_\varepsilon \left[\mathbf{T}(\mu,B_{2R})\right]^{\frac{1}{p-1}}+C(([A]_{R_0})^{\kappa} +\varepsilon)(\fint_{B_{2R}}|\nabla u|^{\gamma_0})^{1/\gamma_0},
	\end{align}
for some $\kappa\in (0,1)$, where
	\begin{equation*}
\mathbf{T}(\mu,B_{2R}):=\left\{ \begin{array}{l}
	\frac{|\mu|(B_{2R})}{R^{n-1}} ~~~~~~~~~\text{if}~~\frac{3n-2}{2n-1}<p\leq 2-\frac{1}{n},\\ 
	\left[\frac{\int_{B_{2R}}|\mu|^\sigma}{R^{n-\sigma}}\right]^{\frac{1}{\sigma }}~~\text{ if }~~1<p\leq \frac{3n-2}{2n-1}. \\ 
	\end{array} \right.
	\end{equation*}     
\end{proposition}
\begin{proof} By \cite[Lemma 2.3 and Corollary 2.4]{55Ph2}, there exists $v\in W^{1,p}(B_R)\cap W^{1,\infty}(B_{R/2})$ such that 
	\begin{align*}
	||\nabla v||_{L^\infty(B_{R/2})}\leq C \left(\fint_{B_R}|\nabla w|^p\right)^{1/p}
	\end{align*}
	and 
	\begin{align*}
	\fint_{B_{R}}|\nabla w-\nabla v|dx\leq C ([A]_{R_0})^{\kappa} \left(\fint_{B_R}|\nabla w|^p\right)^{1/p}
	\end{align*}
	 for some $\kappa\in (0,1).$ Combining these with \eqref{111120148} in Lemma \ref{111120147} and \eqref{1111201410} in Lemma \ref{111120149}, \eqref{1111201410'} in Lemma \ref{111120149'}, we get the results. The proof is complete. 
\end{proof}\medskip\\
      Next, we focus on the corresponding estimates near the boundary. We recall that $\Omega$ is $(\delta_0,R_0)-$ Reifenberg flat domain with $\delta_0<1/2$. Fix $x_0\in \Omega$ and $0<R<R_0/10$. With $u\in W_0^{1,p}(\Omega)$ being a solution to \eqref{5hh070120148}, 
      one considers the unique solution 
      \begin{equation*}
      w\in W_{0}^{1,p}(\Omega_{10R}(x_0))+u
      \end{equation*}
      to the following equation 
      \begin{equation}
      \label{111120146*}\left\{ \begin{array}{l}
      - \operatorname{div}\left( {A(x,\nabla w)} \right) = 0 \;in\;\Omega_{10R}(x_0), \\ 
      w = u\quad \quad on~~\partial \Omega_{10R}(x_0), \\ 
      \end{array} \right.
      \end{equation}
      Hereafter, the notation $\Omega_r(x)$ indicates the set $\Omega\cap B_r(x)$. 
      By \cite[Lemma 2.5]{55Ph0}, we have 
      \begin{lemma} \label{111120147*} Let $w$ be in \eqref{111120146*}.
      	There exist  constants $\theta_1>p$ and $C$ depending only on $N,p,\Lambda$ such that the following estimate      
      	\begin{equation}\label{111120148*}
      	\left(\fint_{B_{\rho/2}(y)}|\nabla w|^{\theta_1} dxdt\right)^{\frac{1}{\theta_1}}\leq C\left(\fint_{B_{3\rho}(y)}|\nabla w|^{p-1} dx\right)^{\frac{1}{p-1}},
      	\end{equation}holds 
      	for all  $B_{3\rho}(y)\subset B_{10R}(x_0)$. 
      \end{lemma}
  As Lemma  \ref{111120149} and \ref{111120149'},  we also get
  \begin{lemma}\label{111120149"} Assume that $p>\frac{3n-2}{2n-1}$. Let $w$ be in \eqref{111120146*} and $\gamma_0$ be in Lemma  \ref{111120149}. There holds 
  	\begin{align}\nonumber
  &	\left(	\fint_{B_{10R}(x_0)}|\nabla (u-w)|^{\gamma_0}dx\right)^{\frac{1}{\gamma_0}}\\&~~~~\leq C\left[\frac{|\mu|(B_{10R}(x_0))}{R^{n-1}}\right]^{\frac{1}{p-1}}+C \frac{|\mu|(B_{10R}(x_0))}{R^{n-1}}\left(	\fint_{B_{10R}(x_0)}|\nabla u|^{\gamma_0}dx\right)^{\frac{2-p}{\gamma_0}}\label{1111201410"+}.
  	\end{align}
   In particular, for any $\varepsilon>0$,
  	 	\begin{align}\label{1111201410"}
  	 \left(	\fint_{B_{10R}(x_0)}|\nabla (u- w)|^{\gamma_0}dx\right)^{\frac{1}{\gamma_0}}\leq C_\varepsilon \left[\frac{|\mu|(B_{10R}(x_0))}{R^{n-1}}\right]^{\frac{1}{p-1}}+\varepsilon \left(	\fint_{B_{10R}(x_0)}|\nabla u|^{\gamma_0}dx\right)^{\frac{1}{\gamma_0}}.
  	 \end{align}
  \end{lemma}
\begin{lemma}\label{111120149""} Assume that $1<p\leq \frac{3n-2}{2n-1}$ and $\mu \in L^\sigma(\Omega)$ for  $\sigma>\frac{n}{p(2n-1)-2(n-1)}$. Let $w$ be in \eqref{111120146*} and $\gamma_0$ be in Lemma  \ref{111120149'}.  There holds 
	\begin{align}\nonumber
	&	\left(	\fint_{B_{10R}(x_0)}|\nabla (u- w)|^{\gamma_0}dx\right)^{\frac{1}{\gamma_0}}\\&~~~~\leq C\left[\frac{\int_{B_{10R}(x_0)}|\mu|^\sigma}{R^{n-\sigma}}\right]^{\frac{1}{\sigma (p-1)}}+C \left[\frac{\int_{B_{10R}(x_0)}|\mu|^\sigma}{R^{n-\sigma}}\right]^{\frac{1}{\sigma }}\left(	\fint_{B_{10R}(x_0)}|\nabla u|^{\gamma_0}dx\right)^{\frac{2-p}{\gamma_0}}\label{1111201410'"+}.
	\end{align}
	 In particular, for any $\varepsilon>0$,
	\begin{align}\label{1111201410'"}
	\left(	\fint_{B_{10R}(x_0)}|\nabla (u- w)|^{\gamma_0}dx\right)^{\frac{1}{\gamma_0}}\leq C_\varepsilon \left[\frac{\int_{B_{10R}(x_0)}|\mu|^\sigma}{R^{n-\sigma}}\right]^{\frac{1}{\sigma (p-1)}}+\varepsilon \left(	\fint_{B_{10R}(x_0)}|\nabla u|^{\gamma_0}dx\right)^{\frac{1}{\gamma_0}}.
	\end{align}
\end{lemma}
   \begin{proposition} \label{boundary}Assume that $\mu\in L^\sigma(\Omega)$ for $\sigma>\frac{n}{p(2n-1)-2(n-1)}$ if $1<p\leq \frac{3n-2}{2n-1}$. Let $\gamma_0$ be in Lemma  \ref{111120149} and \ref{111120149'}. For any $\varepsilon>0$, there exists $\delta_0=\delta_0(n,\Lambda,\varepsilon)$ such that the following holds. If $\Omega$ is $(\delta_0,R_0)-$ Reifenberg flat domain and $u\in W_{0}^{1,p}(\Omega)$ with $x_0\in\partial \Omega$ and $0<R<R_0/10$, there exists a function $V\in W^{1,\infty}(B_{R/10}(x_0))$ 
   	\begin{align}
   	||\nabla V||_{L^\infty(B_{R/10}(x_0))}\leq C \left[\mathbf{T}(\mu,B_{10R}(x_0))\right]^{\frac{1}{p-1}}+C (\fint_{B_{10R}}|\nabla u|^{\gamma_0})^{1/\gamma_0}
   	\end{align}
   	and
   	\begin{align}\nonumber
   &	(\fint_{B_{R/10}(x_0)}|\nabla (u-V)|^{\gamma_0}dx)^{\frac{1}{\gamma_0}}\\&~~~~~\leq C_\varepsilon \left[\mathbf{T}(\mu,B_{10R}(x_0))\right]^{\frac{1}{p-1}}+C(([A]_{R_0})^{\kappa} +\varepsilon)(\fint_{B_{10R}(x_0)}|\nabla u|^{\gamma_0})^{1/\gamma_0}
   	\end{align}
   	 for some $\kappa\in (0,1)$, where
   	 \begin{equation*}
   	 \mathbf{T}(\mu,B_{10R}(x_0)):=\left\{ \begin{array}{l}
   	 \frac{|\mu|(B_{10R}(x_0))}{R^{n-1}} ~~~\text{if}~~\frac{3n-2}{2n-1}<p\leq 2-\frac{1}{n},\\ 
   	 \left[\frac{\int_{B_{10R}(x_0)}|\mu|^\sigma}{R^{n-\sigma}}\right]^{\frac{1}{\sigma }}, ~~\text{ if }~1<p\leq \frac{3n-2}{2n-1}. \\ 
   	 \end{array} \right.
   	 \end{equation*}    
   \end{proposition} 
\begin{proof}
By \cite[Corollary 2.13]{55Ph2}, for any $\varepsilon>0$, there exists $\delta_0=\delta_0(n,\Lambda,\varepsilon)$ such that if $\Omega$ is $(\delta_0,R_0)-$ Reifenberg flat domain  then we find  $V\in W^{1,\infty}(B_{R/10}(x_0))$ satisfying  
\begin{align*}
||\nabla V||_{L^\infty(B_{R/10}(x_0))}\leq C \left(\fint_{B_R(x_0)}|\nabla w|^p\right)^{1/p}
\end{align*}
and 
\begin{align*}
\fint_{B_{R/10}(x_0)}|\nabla w-\nabla V|dx\leq C (([A]_{R_0})^{\kappa}+\varepsilon) \left(\fint_{B_R(x_0)}|\nabla w|^p\right)^{1/p}
\end{align*}
for some $\kappa\in (0,1)$. Combining these with \eqref{111120148*} in Lemma \ref{111120147*} and \eqref{1111201410"} in Lemma \ref{111120149"},  \eqref{1111201410'"} in Lemma \ref{111120149""}, we get the results. The proof is complete. 
\end{proof}
\section{Global estimates on Reifenberg flat domains }
Now we prove results for Reifenberg flat domain. First,
we will use proposition \ref{inter}, \ref{boundary} to get the following result.
\begin{theorem}\label{5hh23101312}   Let $w\in A_\infty$, $\mu\in\mathfrak{M}_b(\Omega)$. Assume that $\mu\in L^\sigma(\Omega)$ for $\sigma>\frac{n}{p(2n-1)-2(n-1)}$ if $1<p\leq \frac{3n-2}{2n-1}$.  Let $\gamma_0$ be in Lemma  \ref{111120149} and \ref{111120149'}.  For any $\varepsilon>0,R_0>0$ one finds  $\delta_1=\delta_1(n,p,\sigma,\Lambda,\varepsilon,[w]_{A_\infty})\in (0,1)$ and $\delta_2=\delta_2(n,p,\sigma,\Lambda,\varepsilon,[w]_{A_\infty},diam(\Omega)/R_0)\in (0,1)$ and $\Lambda_0=\Lambda_0(n,p,\sigma,\Lambda)>0$ such that if $\Omega$ is  $(\delta_1,R_0)$- Reifenberg flat domain and $[\mathcal{A}]_{R_0}\le \delta_1$ then  for any renormalized solution $u$ \eqref{5hh070120148}, we have for any $\lambda$
	\begin{equation}\label{5hh16101311}
	w(\{(M(|\nabla u|^{\gamma_0}))^{1/\gamma_0}>\Lambda_0\lambda, (\mathbf{K}_1(\mu))^{\frac{1}{p-1}}\le \delta_2\lambda \}\cap \Omega)\le C\varepsilon w(\{ (M(|\nabla u|^{\gamma_0}))^{1/\gamma_0}> \lambda\}\cap \Omega)
	\end{equation}
	where the constant $C$  depends only on $n,p,\sigma,\Lambda,diam(\Omega)/R_0, [w]_{A_\infty}$ and $	\mathbf{K}_1(\mu)$ is denoted in \eqref{def}.
\end{theorem} 
Hereafter, $M$ denotes the  Hardy-Littlewood maximal function defined for each locally integrable function  $f$ in $\mathbb{R}^{n}$ by
\begin{equation*}
\mathcal{M}(f)(x)=\sup_{\rho>0}\fint_{B_\rho(x,t)}|f(y)|dy~~\forall x\in\mathbb{R}^{N}.
\end{equation*}
If $p>1$  we verify that $M$ is operator from $L^1(\mathbb{R}^{N})$ into $L^{1,\infty}(\mathbb{R}^{N})$  see  \cite{55Stein2,55Stein3}. 
\\ We would like to mention that the use
of the Hardy-Littlewood maximal function in non-linear degenerate problems was started in the elliptic setting by T.  Iwaniec in his fundamental paper \cite{Iwa}.\medskip \\
To prove Theorem \ref{5hh23101312}, we will use  L. Caddarelli and I. Peral's technique in \cite{CaPe}. Namely, it is based on the following technical lemma whose proof is  a consequence of Lebesgue Differentiation Theorem and the standard Vitali covering  lemma, can be found in  \cite{55BW4,55MePh2} with some modifications to fit the setting here.

\begin{lemma}\label{5hhvitali2} Let $\Omega$ be a $(\delta,R_0)$-Reifenberg flat domain with $\delta<1/4$ and let $w$ be an $\mathbf{A}_\infty$ weight. Suppose that the sequence of balls $\{B_r(y_i)\}_{i=1}^L$ with centers $y_i\in\overline{\Omega}$ and  radius $r\leq R_0/4$ covers $\Omega$.  Let $E\subset F\subset \Omega$ be measurable sets for which there exists $0<\varepsilon<1$ such that  $w(E)<\varepsilon w(B_r(y_i))$ for all $i=1,...,L$, and  for all $x\in \Omega$, $\rho\in (0,2r]$, we have
	$B_\rho(x)\cap \Omega\subset F$      
	if $w(E\cap B_\rho(x))\geq \varepsilon w(B_\rho(x))$. Then $
	w(E)\leq \varepsilon Cw(F)$         
	for a constant $C$ depending only on $n$ and $[w]_{\mathbf{A}_\infty}$.
\end{lemma}   
\begin{proof}[Proof of Theorem \ref{5hh23101312}] Let $u$ be a renormalized solution of \eqref{5hh070120148}. By Proposition \ref{pro1}, we have
	\begin{align*}
	||\nabla u||_{L^{\frac{(p-1)n}{n-1},\infty}(\Omega)}\leq C\left[|\mu|(\Omega)\right]^{\frac{1}{p-1}},
	\end{align*}
	which implies, for any $\gamma\in \left(0,\frac{(p-1)n}{n-1}\right)$, 
	\begin{align}\label{es14}
\left(	\frac{1}{T_0^n}\int_{\Omega}|\nabla u|^{\gamma}\right)^{1/\gamma}\leq C_\gamma \left[\frac{|\mu|(\Omega)}{T_0^{n-1}}\right]^{\frac{1}{p-1}}
	\end{align}
	with $T_0=diam(\Omega)$.\\
Let $\mu_0,\lambda_k^+,\lambda_k^-$ be as in Definition \ref{derenormalized}. Let $u_k\in W_0^{1,p}(\Omega)$ be a unique solution of equation
	\begin{equation}
	\left\{
	\begin{array}
	[c]{l}%
	-\text{div}(A(x,\nabla u_k))=\mu_{k}~~\text{in }\Omega,\\
	{u}_{k}=0\qquad\text{on }\partial\Omega,\\
	\end{array}
	\right.  
	\end{equation}
	where 
	\begin{description}
		\item[1.] if  $\frac{3n-2}{2n-1}<p\leq \frac{2n-1}{n}$,  $\mu_k=\mu_0+\lambda_k^+-\lambda_k^-$, in this case, we have  $u_k=T_k(u)$,
		\item[2.]  if $1<p\leq \frac{3n-2}{2n-1}$, $\mu_k=T_k(\mu)$.
	\end{description}
	 Thus,  by Proposition \ref{pro2}      there exists a subsequence $(u_k)_k$ such that  $\nabla u_k\to \nabla u$  in $L^\gamma(\Omega)$ for all $\gamma\in (0,\frac{(p-1)n}{n-1})$. \medskip\\  To estimates \eqref{5hh16101311}, we will follow  \cite[the proof of Theorem 8.4]{55QH2} and see also \cite[The proof of Theorem 3.1]{55QH3}. Set $$E_{\lambda,\delta_2}=\{(M(|\nabla u|^{\gamma_0}))^{1/\gamma_0}>\Lambda_0\lambda, (\mathbf{K}_1(\mu))^{\frac{1}{p-1}}\le \delta_2\lambda \}\cap \Omega,$$ and $$F_\lambda=\{ (M(|\nabla u|^{\gamma_0}))^{1/\gamma_0}> \lambda\}\cap \Omega,$$ for $\varepsilon\in (0,1)$ and $\lambda>0$, where $\Lambda_0$ is a constant depending only on $n,p,\sigma,\gamma_0,\Lambda$, we will choose it later.  
	 Let $\{y_i\}_{i=1}^L\subset \Omega$ and a ball $B_0$ with radius $2T_0$ such that 
	 $$
	 \Omega\subset \bigcup\limits_{i = 1}^L {{B_{r_0}}({y_i})}  \subset {B_0}$$
	 where $r_0=\min\{R_0/1000,T_0\}$. 
	 We verify that
	 \begin{equation}\label{5hh2310131}
	 w(E_{\lambda,\delta_2})\leq \varepsilon w({B_{r_0}}({y_i})) ~~\forall ~\lambda>0
	 \end{equation}
	 for some $\delta_2$ small enough, depended on $n,p,\Lambda,\epsilon,[w]_{\mathbf{A}_\infty},T_0/R_0$.\\
	 In fact, we can assume that $E_{\lambda,\delta_2}\not=\emptyset$ so $|\mu| (\Omega)\leq T_0^{n-1}(\delta_2\lambda)^{p-1}$ if  $\frac{3n-2}{2n-1}<p\leq \frac{2n-1}{n}$ and $||\mu||_{L^{\sigma}(\Omega)}\leq T_0^{\frac{n}{\sigma}-1}(\delta_2\lambda)^{p-1}$ if $1<p\leq \frac{3n-2}{2n-1}$. Since  $M$ is a bounded operator from $L^1(\mathbb{R}^{N})$ into $L^{1,\infty}(\mathbb{R}^{N})$,
	 \begin{align*}
	  |E_{\lambda,\delta_2}|&\leq \frac{C}{(\Lambda_0\lambda)^{\gamma_0}}\int_{\Omega}|\nabla u|^{\gamma_0}dx\\&\leq \frac{CT_0}{(\Lambda_0\lambda)^{\gamma_0}} \left[\frac{|\mu|(\Omega)}{T_0^{n-1}}\right]^{\frac{\gamma_0}{p-1}},
	 \end{align*}
	 here we used \eqref{es14} for $\gamma=\gamma_0$ in the last inequality.
	 Thus, if $\frac{3n-2}{2n-1}<p\leq \frac{2n-1}{n}$
\begin{align*}
 |E_{\lambda,\delta_2}|\leq \frac{CT_0}{(\Lambda_0\lambda)^{\gamma_0}} \left[\frac{T_0^{n-1}(\delta_2\lambda)^{p-1}}{T_0^{n-1}}\right]^{\frac{\gamma_0}{p-1}}=C\delta_2^{\gamma_0}|B_0|
\end{align*}
  and if $1<p\leq \frac{3n-2}{2n-1}$, 
  \begin{align*}
  |E_{\lambda,\delta_2}|\leq \frac{CT_0}{(\Lambda_0\lambda)^{\gamma_0}} \left[\frac{||\mu||_{L^\sigma(\Omega)}}{T_0^{\frac{n}{\sigma}-1}}\right]^{\frac{\gamma_0}{p-1}}\leq  \frac{CT_0}{(\Lambda_0\lambda)^{\gamma_0}} \left[\frac{T_0^{\frac{n}{\sigma}-1}(\delta_2\lambda)^{p-1}}{T_0^{\frac{n}{\sigma}-1}}\right]^{\frac{\gamma_0}{p-1}}=C\delta_2^{\gamma_0}|B_0|.
  \end{align*}
So, 
	 \begin{align*}
	 	w(E_{\lambda,\delta_2})\leq c\left(\frac{|E_{\lambda,\delta_2}|}{|B_{0}|}\right)^\nu w(B_0)\leq C\delta_2^{\nu\gamma_0} w(B_0)
	 \end{align*}
	 where $(c,\nu)$ is a pair of $A_\infty$ constants of $w$. It is known that (see, e.g \cite{55Gra}) there exist $c_1=c_1(N,c,\nu)$ and $\nu_1=\nu_1(N,c,\nu)$ such that 
	 \begin{equation*}
	 \frac{w(B_0)}{w({B_{r_0}}({y_i}))}\leq c_1\left(\frac{|B_0|}{|{B_{r_0}}({y_i})|}\right)^{\nu_1}~~\forall i.
	 \end{equation*}
	 So, 
	 \begin{align*}
	 w(E_{\lambda,\delta_2})\leq C\delta_2^{\nu\gamma_0} \left(\frac{|B_0|}{|{B_{r_0}}({y_i})|}\right)^{\nu_1} w({B_{r_0}}({y_i}))
	 < \varepsilon w({B_{r_0}}({y_i}))~~\forall ~i
	 \end{align*}
	for  $\delta_2$ small enough depending on $n,p,\sigma,\gamma_0,\epsilon,[w]_{\mathbf{A}_\infty},T_0/R_0$. So, we proved \eqref{5hh2310131}.\\
	 Next we verify that for all $x\in \Omega$ and $r\in (0,2r_0]$ and $\lambda>0$ we have
	 $$
	 B_r(x)\cap \Omega\subset F_\lambda,
	 $$
	 if $$
	 w(E_{\lambda,\delta_2}\cap B_r(x))\geq \varepsilon w(B_r(x)),
	 $$
	 for some $\delta_2$ small enough depending on $n,p,\sigma,\gamma_0,\epsilon,[w]_{\mathbf{A}_\infty},T_0/R_0$ . \\
	 Indeed,
	 take $x\in \Omega$ and $0<r\leq 2r_0$.
	 Now assume that $B_r(x)\cap \Omega\cap F^c_\lambda\not= \emptyset$ and $E_{\lambda,\delta_2}\cap B_r(x)\not = \emptyset$ i.e, there exist $x_1,x_2\in B_r(x)\cap \Omega$ such that $\left[M(|\nabla u|^{\gamma_0})(x_1)\right]^{1/\gamma_0}\leq \lambda$ and $\mathbf{K}_1(\mu)(x_2)\le (\delta_2 \lambda)^{p-1}$.
	 We need to prove that
	 \begin{equation}\label{5hh2310133}
	 w(E_{\lambda,\delta_2}\cap B_r(x)))< \varepsilon w(B_r(x)). 
	 \end{equation}
	 Clearly,
	 \begin{equation*}
	 M(|\nabla u|)(y)\leq \max\{\left[M\left(\chi_{B_{2r}(x)}|\nabla u|^{\gamma_0}\right)(y)\right]^{\frac{1}{\gamma_0}},3^{n}\lambda\}~~\forall y\in B_r(x).
	 \end{equation*}
	 Therefore, for all $\lambda>0$ and $\Lambda_0\geq 3^{n}$,
	 \begin{eqnarray}\label{5hh2310134}E_{\lambda,\delta_2}\cap B_r(x)=\{M\left(\chi_{B_{2r}(x)}|\nabla u|^{\gamma_0}\right)^{\frac{1}{\gamma_0}}>\Lambda_0\lambda, (\mathbf{K}_{1}(\mu))^{\frac{1}{p-1}}\leq \delta_2\lambda\}\cap \Omega \cap B_r(x).
	 \end{eqnarray}
	 In particular, $E_{\lambda,\delta_2}\cap B_r(x)=\emptyset$ if $\overline{B}_{8r}(x)\subset\subset \mathbb{R}^{N}\backslash \Omega$.
	 Thus, it is enough to consider the case $B_{8r}(x)\subset\subset\Omega$ and $B_{8r}(x)\cap\Omega\not=\emptyset$.\\   
	 We consider the case $B_{8r}(x)\subset\subset\Omega$. Applying  Proposition \ref{inter} to  $u=u_{k}\in W_{0}^{1,p}(\Omega),\mu=\mu_k$ and $B_{2R}=B_{8r}(x)$, there is a functions $v_k\in W^{1,p}(B_{4r}(x))\cap W^{1,\infty}(B_{2r}(x))$ such that for any $\eta>0$, 
	 \begin{align}
	 ||\nabla v_k||_{L^\infty(B_{2r}(x))}\leq C \left[\mathbf{T}(\mu_k, B_{8r}(x))\right]^{\frac{1}{p-1}}+C (\fint_{B_{8r}(x)}|\nabla u_k|^{\gamma_0})^{1/\gamma_0}
	 \end{align}
	 and
	 \begin{align}
	 (\fint_{B_{4r}}|\nabla u_k-\nabla v_k|^{\gamma_0}dx)^{\frac{1}{\gamma_0}}\leq C_\eta \left[\mathbf{T}(\mu_k,B_{8r}(x))\right]^{\frac{1}{p-1}}+C(([A]_{R_0})^{\kappa} +\eta)(\fint_{B_{8r}}|\nabla u_k|^{\gamma_0})^{1/\gamma_0}
	 \end{align}
	 for some $\kappa\in (0,1)$, where 
	 \begin{equation*}
	 \mathbf{T}(\mu_k,B_{8r}(x)):=\left\{ \begin{array}{l}
	 \frac{|\mu_k|(B_{8r}(x))}{r^{n-1}} ~~~~~~~~~~\text{if}~~\frac{3n-2}{2n-1}<p\leq 2-\frac{1}{n},\\ 
	 \left[\frac{\int_{B_{8r}(x)}|\mu_k|^\sigma}{r^{n-\sigma}}\right]^{\frac{1}{\sigma }} ~~\text{ if }~1<p\leq \frac{3n-2}{2n-1}. \\ 
	 \end{array} \right.
	 \end{equation*}   
	 Thanks to $\left[M(|\nabla u|^{\gamma_0})(x_1)\right]^{1/\gamma_0}\leq \lambda$ and  $[\mathbf{K}_1(\mu)(x_2)]^{\frac{1}{p-1}}\le \delta_2 \lambda$  with $x_1,x_2\in B_r(x)$ and definitions of $u_k,\mu_k$, we get 
	 \begin{align*}
	 \mathop {\limsup }\limits_{k \to \infty } ||\nabla v_{k}||_{L^\infty(B_{2r}(x))}&\leq   C \left[\mathbf{T}(\mu, \overline{B_{8r}(x)})\right]^{\frac{1}{p-1}}+C (\fint_{B_{8r}(x)}|\nabla u|^{\gamma_0})^{1/\gamma_0}    \\&\leq   C [\mathbf{K}_1(\mu)(x_2)]^{\frac{1}{p-1}}+C \left[M(|\nabla u|^{\gamma_0})(x_1)\right]^{1/\gamma_0}                     
	 \\&\leq C\lambda,  
	 \end{align*}
	 and 
	 \begin{align*}
	 \mathop {\limsup }\limits_{k \to \infty }  (\fint_{B_{4r}(x)}|\nabla u_k-\nabla v_k|^{\gamma_0}dx)^{\frac{1}{\gamma_0}}&\leq C_\eta \left[\mathbf{T}(\mu,\overline{B_{8r}(x)})\right]^{\frac{1}{p-1}}+C(([A]_{R_0})^{\kappa} +\eta)\left[M(|\nabla u|^{\gamma_0})(x_1)\right]^{1/\gamma_0}                   \\
	 &\leq C_\eta [\mathbf{K}_1(\mu)(x_2)]^{\frac{1}{p-1}}+C(([A]_{R_0})^{\kappa} +\eta)(\fint_{B_{8r}(x)}|\nabla u|^{\gamma_0})^{1/\gamma_0}
	 \\&\leq C\left(C_{\eta}\delta_2+\delta_1^{\kappa}+\eta\right)\lambda.
	 \end{align*}                                                                Here we used $[A]_{R_0}\leq \delta_1$ in the last inequality. \\
	 So, we can assume that 
	 \begin{equation}\label{5hh2310136}
	 ||\nabla v_{k}||_{L^\infty(B_{2r}(x))}\leq C\lambda ~~\text{and}
	 \end{equation}                        
	 \begin{equation}\label{5hh2310137} 
	 (\fint_{B_{4r}(x)}|\nabla u_k-\nabla v_k|^{\gamma_0}dx)^{\frac{1}{\gamma_0}}\leq C\left(C_{\eta}\delta_2+\delta_1^{\kappa}+\eta\right)\lambda
	 \end{equation}
	 for all $k\geq k_0$.\\
	  Since 
	  \begin{align*}
	  (	M(|\sum_{j=1}^{m}f_j|^{\gamma_0}))^{1/\gamma_0}\leq m\sum_{j=1}^{m}	(M(|f_j|^{\gamma_0}))^{1/\gamma_0}
	  \end{align*}       
	  for all $m\geq2$, 
	  \begin{align}\nonumber
	  |E_{\lambda,\delta_2}\cap B_r(x)|&\leq   |\{M\left(\chi_{B_{2r}(x)}|\nabla (u_k-v_k)|^{\gamma_0}\right)^{\frac{1}{\gamma_0}}>\Lambda_0\lambda/9\}\cap B_r(x)|
	  \\&\nonumber+ |\{M\left(\chi_{B_{2r}(x)}|\nabla (u-u_k)|^{\gamma_0}\right)^{\frac{1}{\gamma_0}}>\Lambda_0\lambda/9\}\cap B_r(x)|\\&+
	  |\{M\left(\chi_{B_{2r}(x)}|\nabla v_k|^{\gamma_0}\right)^{\frac{1}{\gamma_0}}>\Lambda_0\lambda/9\}\cap B_r(x)|         \label{es18}                 
	  \end{align}  
	 In view of \eqref{5hh2310136} we see that for $\Lambda_0\geq \max\{3^{n},10C\}$ ($C$ is the constant in \eqref{5hh2310136} ) and $k\geq k_0$,
	 \begin{align*}
	 |\{M\left(\chi_{B_{2r}(x)}|\nabla v_k|^{\gamma_0}\right)^{\frac{1}{\gamma_0}}>\Lambda_0\lambda/9\}\cap B_r(x)|=0.
	 \end{align*}
Thus, we deduce from \eqref{es18} and \eqref{5hh2310137} that 
	 \begin{align*}
	 |E_{\lambda,\delta_2}\cap B_r(x)|&\leq   |\{M\left(\chi_{B_{2r}(x)}|\nabla (u_k-v_k)|^{\gamma_0}\right)^{\frac{1}{\gamma_0}}>\Lambda_0\lambda/9\}\cap B_r(x)|
	 \\&+ |\{M\left(\chi_{B_{2r}(x)}|\nabla (u-u_k)|^{\gamma_0}\right)^{\frac{1}{\gamma_0}}>\Lambda_0\lambda/9\}\cap B_r(x)|\\&\leq \frac{C}{\lambda^{\gamma_0}}    \left(\int_{B_{2r}(x)}  |\nabla (u_k-v_k)|^{\gamma_0}+ \int_{B_{2r}(x)}  |\nabla (u-u_k)|^{\gamma_0} \right)    
	 \\&\leq \frac{C}{\lambda^{\gamma_0}}    \left(\left(C_{\eta}\delta_2+\delta_1^{\kappa}+\eta\right)^{\gamma_0}\lambda^{\gamma_0}r^n+ \int_{B_{2r}(x)}  |\nabla (u-u_k)|^{\gamma_0} \right)    
	 \end{align*} 
	 Letting $k\to\infty$  we get 
	 \begin{equation*}
	 |E_{\lambda,\delta_2}\cap B_r(x)|\leq C \left(C_{\eta}\delta_2+\delta_1^{\kappa}+\eta\right)^{\gamma_0}|B_r(x)|.
	 \end{equation*}
	 Thus,  
	 \begin{align*}
	 w(E_{\lambda,\delta_2}\cap B_r(x))&\leq c\left(\frac{|E_{\lambda,\delta_2}\cap B_r(x) |}{|B_r(x)|}\right)^\nu w(B_r(x))
	 \\&\leq  c\left(C_{\eta}\delta_2+\delta_1^{\kappa}+\eta\right)^{\gamma_0\nu} w(B_r(x))
	 \\&< \varepsilon w(B_r(x)).
	 \end{align*} 
	 where $\eta,\delta_1,\delta_2\leq C(n,p,\sigma,\gamma_0,\epsilon,[w]_{\mathbf{A}_\infty})$.\\
	 Next we consider the case $B_{8r}(x)\cap\Omega\not=\emptyset$. Let $x_3\in\partial \Omega$ such that $|x_3-x|=\text{dist}(x,\partial\Omega)$.  We have 
	 \begin{equation}\label{5hh2310138}
	 B_{2r}(x)\subset Q_{10r}(x_3)\subset B_{100r}(x_3)\subset B_{108r}(x)\subset B_{109r}(x_1)
	 \end{equation}
	 and 
	 \begin{equation}\label{5hh2310139}
	 B_{100r}(x_3)\subset B_{108r}(x)\subset B_{109r}(x_2)
	 \end{equation}
	  Applying  Proposition \ref{boundary} to  $u=u_{k}\in W_{0}^{1,p}(\Omega),\mu=\mu_k$ and $B_{10R}=B_{100r}(x_3)$, 	for any $\eta>0$, there exists $\delta_0=\delta_0(n,p,\sigma,\Lambda,\eta)$ such that the following holds. If $\Omega$ is $(\delta_0,R_0)-$ Reifenberg flat domain, there exists a function $V_k\in W^{1,\infty}(B_{10r}(x_3))$  such that 
	  \begin{align}
	  ||\nabla V_k||_{L^\infty(B_{10r}(x_3))}\leq C \left[\mathbf{T}(\mu_k,B_{100r}(x_3))\right]^{\frac{1}{p-1}}+C (\fint_{B_{100r}(x_3)}|\nabla u_k|^{\gamma_0})^{1/\gamma_0}
	  \end{align}
	  \begin{align}\nonumber &
	  (\fint_{B_{10r}(x_3)}|\nabla (u_k-V_k)|^{\gamma_0}dx)^{\frac{1}{\gamma_0}}\\&~~~~~~~~\leq C_\eta \left[\mathbf{T}(\mu_k,B_{100r}(x_3))\right]^{\frac{1}{p-1}}+C(([A]_{R_0})^{\kappa} +\eta)(\fint_{B_{100r}(x_3)}|\nabla u_k|^{\gamma_0})^{1/\gamma_0}
	  \end{align}
	  for some  $\kappa\in (0,1)$, where  
	  \begin{equation*}
	  \mathbf{T}(\mu,B_{100r}(x_3)):=\left\{ \begin{array}{l}
	  \frac{|\mu|(B_{100r}(x_3))}{r^{n-1}} ~~~~~~~~~\text{if}~~\frac{3n-2}{2n-1}<p\leq 2-\frac{1}{n},\\ 
	  \left[\frac{\int_{B_{100r}(x_3)}|\mu|^\sigma}{r^{n-\sigma}}\right]^{\frac{1}{\sigma }}~~\text{ if }~1<p\leq \frac{3n-2}{2n-1}. \\ 
	  \end{array} \right.
	  \end{equation*}    
	 Since $\left[M(|\nabla u|^{\gamma_0})(x_1)\right]^{1/\gamma_0}\leq \lambda$ and  $[\mathbf{K}_1(\mu)(x_2)]^{\frac{1}{p-1}}\le \delta_2 \lambda$  with $x_1,x_2\in B_r(x)$ and \eqref{5hh2310138}, \eqref{5hh2310139} and  and definitions of $u_k,\mu_k$,  we get 
	 \begin{align*}
	 \mathop {\limsup }\limits_{k \to \infty } ||\nabla V_k||_{L^\infty(B_{2r}(x))}&\leq C \left[\mathbf{T}(\mu,\overline{B_{100r}(x_3)})\right]^{\frac{1}{p-1}}+C (\fint_{B_{100r}(x_3)}|\nabla u|^{\gamma_0})^{1/\gamma_0}
	 \\& \leq C \left[\mathbf{T}(\mu,B_{109r}(x_2))\right]^{\frac{1}{p-1}}+C (\fint_{B_{109r}(x_1)}|\nabla u|^{\gamma_0})^{1/\gamma_0}
	 \\&\leq C\left([\mathbf{K}_1(\mu)(x_2)]^{\frac{1}{p-1}}+\left[M(|\nabla u|^{\gamma_0})(x_1)\right]^{1/\gamma_0}\right)
	 \\&\leq C\lambda,
	 \end{align*}
	 and 
	 \begin{align*}
	 &\mathop {\limsup }\limits_{k \to \infty }(\fint_{B_{2r}(x)}|\nabla (u_k-V_k)|^{\gamma_0}dx)^{\frac{1}{\gamma_0}}\\&~~~~~~~~\leq C_\eta [\mathbf{K}_1(\mu)(x_2)]^{\frac{1}{p-1}}+C(([A]_{R_0})^{\kappa} +\eta)\left[M(|\nabla u|^{\gamma_0})(x_1)\right]^{1/\gamma_0}
	 \\&~~~~~~~~\leq C\left( C_\eta \delta_2+\delta_1^{\kappa}+\eta\right)\lambda.
	 \end{align*}
	 Here we used $[A]_{R_0}\leq \delta_1$ in the last inequality.\\
	 So, we can assume that 
	 \begin{align}\label{5hh23101310}
	 ||\nabla V_k||_{L^\infty(B_{2r}(x))}\leq C\lambda,
	 \end{align}
	 and 
	 \begin{align}\label{5hh23101311}
	 (\fint_{B_{2r}(x)}|\nabla (u_k-V_k)|^{\gamma_0}dx)^{\frac{1}{\gamma_0}}\leq C\left( C_\eta \delta_2+\delta_1^{\kappa}+\eta\right)\lambda.
	 \end{align}
	 for all $k\geq k_0$.\\
	  As above we also have  for $k\geq k_0$  \begin{align*}
	  |E_{\lambda,\delta_2}\cap B_r(x)|&\leq   |\{M\left(\chi_{B_{2r}(x)}|\nabla (u_k-v_k)|^{\gamma_0}\right)^{\frac{1}{\gamma_0}}>\Lambda_0\lambda/9\}\cap B_r(x)|
	  \\&+ |\{M\left(\chi_{B_{2r}(x)}|\nabla (u-u_k)|^{\gamma_0}\right)^{\frac{1}{\gamma_0}}>\Lambda_0\lambda/9\}\cap B_r(x)|.
	  \end{align*}                      
	  for some constant $\Lambda_0$ depending only on $n,p,\sigma,\Lambda$. 
	 Therefore, we deduce from \eqref{5hh23101310} and \eqref{5hh23101311} that 
	 \begin{align*}
	 |E_{\lambda,\delta_2}\cap B_r(x)|&\leq  \frac{C}{\lambda^{\gamma_0}}    \left(\int_{B_{2r}(x)}  |\nabla (u_k-v_k)|^{\gamma_0}+ \int_{B_{2r}(x)}  |\nabla (u-u_k)|^{\gamma_0} \right)    
	 \\&\leq \frac{C}{\lambda^{\gamma_0}}    \left(\left(C_{\eta}\delta_2+\delta_1^{\kappa}+\eta\right)^{\gamma_0}\lambda^{\gamma_0}r^n+ \int_{B_{2r}(x)}  |\nabla (u-u_k)|^{\gamma_0} \right)    
	 \end{align*} 
	 Letting $k\to\infty$  we get 
	\begin{equation*}
	|E_{\lambda,\delta_2}\cap B_r(x)|\leq C \left(C_{\eta}\delta_2+\delta_1^{\kappa}+\eta\right)^{\gamma_0}|B_r(x)|.
	\end{equation*}
	 Thus\begin{align*}
	 w(E_{\lambda,\delta_2}\cap B_r(x))&\leq c\left(\frac{|E_{\lambda,\delta_2}\cap B_r(x) |}{|B_r(x)|}\right)^\nu w(B_r(x))
	 \\&\leq  c\left(C_{\eta}\delta_2+\delta_1^{\kappa}+\eta\right)^{\gamma_0\nu} w(B_r(x))
	 \\&< \varepsilon w(B_r(x)).
	 \end{align*} 
	 where $\eta,\delta_1,\delta_2\leq C(n,p,\sigma,\gamma_0,\varepsilon,[w]_{\mathbf{A}_\infty})$.\\
	 Therefore, for all $x\in \Omega$ and $r\in (0,2r_0]$ and $\lambda>0$ if  
	 $$w(E_{\lambda,\delta_2}\cap B_r(x))\geq \varepsilon w(B_r(x))$$              
	 then            
	 $$ B_r(x)\cap \Omega\subset F_\lambda$$  
	 where $\delta_1=\delta_1(n,p,\sigma,\Lambda,\varepsilon,[w]_{A_\infty})\in (0,1)$ and $\delta_2=\delta_2(n,p,\sigma,\Lambda,\varepsilon,[w]_{A_\infty},T_0/R_0)\in (0,1)$. Applying Lemma \ref{5hhvitali2} we get the result.   The proof is complete.                         
\end{proof} \medskip\\\\
\begin{proof}[Proof of Theorem \ref{101120143}]By Theorem \ref{5hh23101312}, for any $\varepsilon>0,R_0>0$ one finds  $\delta=\delta(n,p,\sigma,\Lambda,\varepsilon,[w]_{A_\infty})\in (0,1/2)$ and $\delta_2=\delta_2(n,p,\sigma,\Lambda,\varepsilon,[w]_{A_\infty},diam(\Omega)/R_0)\in (0,1)$ and $\Lambda_0=\Lambda_0(n,p,\sigma,\Lambda)>0$ such that if $\Omega$ is  a $(\delta,R_0)$- Reifenberg flat domain and $[A]_{R_0}\le \delta$ then 
		\begin{equation}\label{1411201413}
	w(\{(M(|\nabla u|^{\gamma_0}))^{1/\gamma_0}>\Lambda_0\lambda, (\mathbf{K}_1(\mu))^{\frac{1}{p-1}}\le \delta_2\lambda \}\cap \Omega)\le C\varepsilon w(\{ (M(|\nabla u|^{\gamma_0}))^{1/\gamma_0}> \lambda\}\cap \Omega)
	\end{equation}
	for all $\lambda>0$, 
	where the constant $\gamma_0$ is in Lemma  \ref{111120149} and \ref{111120149'},  the constant $C$  depends only on $n,p,\sigma,\Lambda,[w]_{A_\infty},diam(\Omega)/R_0$.
	Thus, for $s<\infty,$
	\begin{align*}
&	||(M(|\nabla u|^{\gamma_0}))^{1/\gamma_0}||_{L^{q,s}_w(\Omega)}^s
	=q\Lambda_0^s\int_{0}^{\infty}\lambda^s\left(w(\{(M(|\nabla u|^{\gamma_0}))^{1/\gamma_0}>\Lambda\lambda \}\cap \Omega)\right)^{s/q}\frac{d\lambda}{\lambda} \\&~~~\leq 
	q\Lambda_0^s2^{s/q}(C\varepsilon)^{s/q}\int_{0}^{\infty}\lambda^s\left(w(\{(M(|\nabla u|^{\gamma_0}))^{1/\gamma_0}>\lambda \}\cap \Omega)\right)^{s/q}\frac{d\lambda}{\lambda} 
	\\&~~~~~+  q\Lambda_0^s2^{s/q}\int_{0}^{\infty}\lambda^s\left(w(\{(\mathbf{K}_1(\mu))^{\frac{1}{p-1}}>\delta_2\lambda \}\cap \Omega)\right)^{s/q}\frac{d\lambda}{\lambda}
	\\&~~~ = \Lambda_0^s2^{s/q}(C\varepsilon)^{s/q}||(M(|\nabla u|^{\gamma_0}))^{1/\gamma_0}||_{L^{q,s}_w(\Omega)}^s+\Lambda_0^s2^{s/q}\delta_2^{-s}||\mathcal(\mathbf{K}_1(\mu))^{\frac{1}{p-1}}||_{L^{q,s}_w(\Omega)}^s.
	\end{align*}
	It implies
	\begin{align*}
&	||(M(|\nabla u|^{\gamma_0}))^{1/\gamma_0}||_{L^{q,s}_w(\Omega)}\\&~~\leq 2^{1/s}\Lambda_02^{1/q}(C\varepsilon)^{1/q}||(M(|\nabla u|^{\gamma_0}))^{1/\gamma_0}||_{L^{q,s}_w(\Omega)}+2^{1/s}\Lambda_0 2^{1/q}\delta_2^{-1}||(\mathbf{K}_1(\mu))^{\frac{1}{p-1}}||_{L^{q,s}_w(\Omega)}
	\end{align*}
	and this inequalities  is also true when $s=\infty$. \\
	We can choose $\varepsilon=\varepsilon(n,p,\sigma,\Lambda,s,q,C)>0$ such that   $2^{1/s}\Lambda2^{1/q}(B\varepsilon)^{1/q}\leq 1/2$, then we get the result. The proof is complete.
\end{proof}\medskip\\  
The following is an equi-integrable property of renormalized solutions of \eqref{5hh070120148} . 
\begin{corollary} Let $\mathcal{F}\subset L^1_w(\Omega)$ be a bounded and equi-integrable with $w\in \mathbf{A}_{\infty}$. Let  $(u_k)_k$ be a renormalized solution to \eqref{5hh070120148} with data $\mu=\mu_k \in \mathfrak{M}_b(\Omega)$. Assume that $\mu_k\in L^\sigma(\Omega)$ for $\sigma>\frac{n}{p(2n-1)-2(n-1)}$ if $1<p\leq \frac{3n-2}{2n-1}$ and 
	\begin{align*}
	\left[\mathbf{K}_1(\mu_k)\right]^{\frac{q}{p-1}}\subset \mathcal{F}
	\end{align*}
	for $q>0$.	Then,  we find  $\delta=\delta(n,p,\sigma,\Lambda, q, [w]_{\mathbf{A}_{\infty}})\in (0,1)$ such that if $\Omega$ is  $(\delta,R_0)$-Reifenberg flat domain $\Omega$ and $[A]_{R_0}\le \delta$ for some $R_0>0$, then $(|\nabla u_k|^q)_k$ is a bounded and equi-integrable in   $L^1_{w}(\Omega)$
\end{corollary} 
\begin{proof} Since  $\mathcal{F}$ is a bounded and equi-integrable in $ L^1_w(\Omega)$, by \cite[Proposition 1.27]{luigi} we can find a nondecreasing function $\phi: [0,\infty)\to [0,\infty)$ such that  $\phi(t)\to\infty$ as  $t\to \infty$
	\begin{align*}
	\mathcal{F}\subset \left\{f: \int_{0}^{\infty}\phi(t)w(\{x:|f|>t\}) dt\leq 1\right\}
	\end{align*}
	We can assume that $\phi(2t)\leq 2\phi(t)$ for all $t>1$.\\Thus, 
	\begin{align*}
	\int_{0}^{\infty}\phi(t)w(\{x:	\left[\mathbf{K}_1(\mu_k)\right]^{\frac{q}{p-1}}>t\}) dt\leq 1.
	\end{align*}
	By Theorem \ref{5hh23101312}, for any $\varepsilon>0,R_0>0$ one finds  $\delta=\delta(n,p,\sigma,\Lambda,\varepsilon,[w]_{A_\infty})\in (0,1/2)$ and $\delta_2=\delta_2(n,p,\sigma,\Lambda,\varepsilon,[w]_{A_\infty},diam(\Omega)/R_0)\in (0,1)$ and $\Lambda_0=\Lambda_0(n,p,\sigma,\gamma_0,\Lambda)>0$ such that if $\Omega$ is  a $(\delta,R_0)$- Reifenberg flat domain and $[A]_{R_0}\le \delta$ then 
	\begin{equation}\label{es19}
	w(\{(M(|\nabla u_k|^{\gamma_0}))^{1/\gamma_0}>\Lambda_0\lambda, (\mathbf{K}_1(\mu_k))^{\frac{1}{p-1}}\le \delta_2\lambda \}\cap \Omega)\le C\varepsilon w(\{ (M(|\nabla u_k|^{\gamma_0}))^{1/\gamma_0}> \lambda\}\cap \Omega)
	\end{equation}
	for all $\lambda>0$, 
	where the constant $\gamma_0$ is in Lemma  \ref{111120149} and \ref{111120149'},  the constant $C$  depends only on $n,p,\sigma,\gamma_0,\Lambda,\varepsilon,[w]_{A_\infty},diam(\Omega)/R_0$. 
Thus, 	
\begin{align*}
	&w(\{(M(|\nabla u_k|^{\gamma_0}))^{q/\gamma_0}>t\}\cap \Omega)\\&~~~\le 	w(\{ (\mathbf{K}_1(\mu_k))^{\frac{q}{p-1}}> \left(\frac{\delta_2}{\Lambda_0}\right)^qt \}\cap \Omega)+ C\varepsilon w(\{ (M(|\nabla u_k|^{\gamma_0}))^{q/\gamma_0}> t/\Lambda_0^q\}\cap \Omega)
\end{align*}
for all $t>0$. Therefore, 
	\begin{align*}
&	\int_{0}^{\infty}\phi(t)w(\{x:	(M(|\nabla u_k|^{\gamma_0}))^{\frac{q}{\gamma_0}}>t\}) dt\\&~\leq C\varepsilon \int_{0}^{\infty}\phi(t)w(\{x:	(M(|\nabla u_k|^{\gamma_0}))^{\frac{q}{\gamma_0}}>t/\Lambda_0^q\}) dt \\&~~~~+   \int_{0}^{\infty}\phi(t)w(\{ (\mathbf{K}_1(\mu_k))^{\frac{q}{p-1}}> \left(\frac{\delta_2}{\Lambda_0}\right)^qt \}\cap \Omega)dt\\&\leq C\Lambda_0^{3q}\varepsilon \int_{0}^{\infty}\phi(t)w(\{x:	(M(|\nabla u_k|^{\gamma_0}))^{\frac{q}{\gamma_0}}>t/\Lambda_0^q\}) dt \\&~~~+   C_{\delta_2}\int_{0}^{\infty}\phi(t)w(\{ (\mathbf{K}_1(\mu_k))^{\frac{q}{p-1}}> t \}\cap \Omega)dt
	\end{align*}
which implies 
\begin{align*}
\int_{0}^{\infty}\phi(t)w(\{x:	(M(|\nabla u_k|^{\gamma_0}))^{\frac{q}{\gamma_0}}>t\}) dt&\leq C \int_{0}^{\infty}\phi(t)w(\{ (\mathbf{K}_1(\mu_k))^{\frac{q}{p-1}}> t \}\cap \Omega)dt\\&\leq C
\end{align*}
for $\varepsilon>0$ small enough, ( it does not depend on $\phi$). Hence, 
\begin{align*}
\int_{0}^{\infty}\phi(t)w(\{x:	|\nabla u_k|^{q}>t\}) dt\leq C
\end{align*}
and 
$(|\nabla u_k|^q)_k$ is a bounded and equi-integrable in   $L^1_{w}(\Omega)$. The proof is complete. 
\end{proof}

\end{document}